\documentclass[12pt,a4paper]{amsart}

\input xy
\xyoption{all}

\DeclareMathAlphabet{\mathpzc}{OT1}{pzc}{m}{it}

\usepackage{amsfonts,amsmath,amssymb,indentfirst,mathrsfs,amscd}
\usepackage[mathscr]{eucal}
\usepackage[active]{srcltx}
\usepackage{graphicx}
\usepackage{float}

\usepackage[dvips]{color}
\usepackage{color}

\author[M. Logares]{Marina Logares}
\address{Instituto de Ciencias Matem\'aticas (CSIC-UAM-UC3M-UCM),
C/ Nicolas Cabrera 15, 28049 Madrid, Spain}
\email{marina.logares@icmat.es}

\author[V. Mu\~{n}oz]{Vicente Mu\~{n}oz}
\address{Facultad de Matem\'aticas, Universidad Complutense de Madrid,
Plaza Ciencias 3, 28040 Madrid, Spain}
\address{Instituto de Ciencias Matem\'aticas (CSIC-UAM-UC3M-UCM),
C/ Nicolas Cabrera 15, 28049 Madrid, Spain}
\email{vicente.munoz@mat.ucm.es}

\title[Hodge polynomials of $\SL(2,\CC)$-character varieties]{Hodge 
polynomials of the $\SL(2,\CC)$-character variety of an elliptic curve with two marked points}

\thanks{This work has been supported by MINECO: ICMAT Severo Ochoa project SEV-2011-0087. The author was partially supported by FCT(Portugal) through Project PTDC/MAT/099275/2008.  Partially supported by FCT  with European Regional Development Fund (COMPETE) and national funds through the project PTDC/MAT/098770/2008. Second author partically supported by MICINN Project MTM2010-17389.}

\subjclass[2000]{Primary: 14D20. Secondary: 14C30, 14L24}

\keywords{Character variety, elliptic curve, Hodge-Deligne polynomial}

\DeclareMathOperator{\Id}{Id}             %End%
\DeclareMathOperator{\tr}{Tr\,}             %Trace Tr%
\DeclareMathOperator{\SL}{SL}

\DeclareMathOperator{\PGL}{PGL}

\DeclareMathOperator{\Gr}{Gr}

\DeclareMathOperator{\Tr}{Tr\,}       %trace of a matrix%
\begin{document}

\newtheorem{thm}{Theorem}[section]
\newtheorem{prop}[thm]{Proposition}
\newtheorem{lem}[thm]{Lemma}
\newtheorem{cor}[thm]{Corollary}
\newtheorem{conjecture}{Conjecture}

\theoremstyle{definition}
\newtheorem{defn}[thm]{Definition}
\newtheorem{ex}[thm]{Example}
\newtheorem{as}{Assumption}

\theoremstyle{remark}
\newtheorem{rmk}[thm]{Remark}

\theoremstyle{remark}
\newtheorem*{prf}{Proof}

\newcommand{\iacute}{\'{\i}} %i con acento%
\newcommand{\norm}[1]{\lVert#1\rVert} %norma%

\newcommand{\lto}{\longrightarrow}
\newcommand{\hra}{\hookrightarrow}

\newcommand{\suchthat}{\;\;|\;\;}
\newcommand{\dbar}{\overline{\partial}}

\newcommand{\cC}{\mathcal{C}}
\newcommand{\cD}{\mathcal{D}}
\newcommand{\cG}{\mathcal{G}} %Gauge group%
\newcommand{\cF}{\mathcal{F}}
\newcommand{\cO}{\mathcal{O}} %Holomorphic functions sheaf%
\newcommand{\cM}{\mathcal{M}} %Moduli space% %moduli of parabolic bundles% %moduli of U(p,q) bundles%
\newcommand{\cN}{\mathcal{N}} %Space of minimal points of the Morse function%
\newcommand{\cP}{\mathcal{P}} %Moduli of K(D) pairs%
\newcommand{\cS}{\mathcal{S}} %Moduli of solutions of Hitchin's equations, contructed by Konno%
\newcommand{\cU}{\mathcal{U}} %Moduli of stable U(p,q) parabolic Higgs bundles%
\newcommand{\cX}{\mathcal{X}}
\newcommand{\cT}{\mathcal{T}}
\newcommand{\cV}{\mathcal{V}}
\newcommand{\cB}{\mathcal{B}}
\newcommand{\cR}{\mathcal{R}}
\newcommand{\cH}{\mathcal{H}}

\newcommand{\ext}{\mathrm{ext}} % an extension%
\newcommand{\x}{\times}

\newcommand{\mM}{\mathscr{M}} %Meromorphic function sheaf%

\newcommand{\CC}{\mathbb{C}} %Complex numbers%
\newcommand{\QQ}{\mathbb{Q}} %Rational numbers%
\newcommand{\PP}{\mathbb{P}} %projective space%
\newcommand{\HH}{\mathbb{H}} %Hypercohomology, quaternions..%
\newcommand{\RR}{\mathbb{R}} %Real numbers%
\newcommand{\ZZ}{\mathbb{Z}} %Integer numbers%

\renewcommand{\lg}{\mathfrak{g}} %Lie algebra of G%
\newcommand{\lh}{\mathfrak{h}} %Lie algebra of H%
\newcommand{\lu}{\mathfrak{u}} %Lie algebra of U%
\newcommand{\la}{\mathfrak{a}} %Lie algebra of A%
\newcommand{\lb}{\mathfrak{b}} %Lie algebra of B%
\newcommand{\lm}{\mathfrak{m}} %Lie algebra of M%
\newcommand{\lgl}{\mathfrak{gl}} %Lie algebra of GL%
\newcommand{\too}{\longrightarrow}
\newcommand{\imat}{\sqrt{-1}} %i%

\hyphenation{mul-ti-pli-ci-ty}

\hyphenation{mo-du-li}

\begin{abstract}
We compute the Hodge polynomials for the moduli space of representations of an elliptic curve with two marked points into $\SL(2,\CC)$. 
When we fix the conjugacy classes of the representations around the marked points to be diagonal and of modulus one, the character variety is diffeomorphic 
to the moduli space of strongly parabolic Higgs bundles, whose Betti numbers are known. In that case we can recover some of the Hodge 
numbers of the character variety. We extend this result to the moduli space of doubly periodic instantons.
\end{abstract}

\maketitle

%%%%%%%%%%%%%%%%%%%%%%%%%%%%%%%%%%%%%%%%%%%%%%%%%%%%%%%
\section{Introduction}
%%%%%%%%%%%%%%%%%%%%%%%%%%%%%%%%%%%%%%%%%%%%%%%%%%%%%%%

Let $X$ be a projective algebraic curve of genus $g\geq 1$ and
$x_1,\ldots, x_s\in X$ a collection of marked points. Let $G$ be
a complex reductive Lie group. The $G$-character variety of $X$
with marked points $x_i$ is defined as the moduli space of
representations of $\pi_1(X-\{x_1,\ldots, x_s\})$ into $G$.
For this, we fix conjugacy classes $\cC_1,\ldots, \cC_s\subset G$. The
$G$-character variety is the space
\begin{align*}
  \cR_{\cC_1,\ldots, \cC_s} (X,G)  =  
\Big\{(A_1,B_1,\ldots, A_g,B_g, & C_1,\ldots, C_s)\in G^{2g+s} \,| \\ 
 & \prod_{i=1}^g [A_i,B_i] \prod_{j=1}^s C_j=\Id, 
 C_j\in \cC_j, 1\leq j \leq s  \Big\}// G,
\end{align*}
where $G$ acts by simultaneous conjugation. This is the space
of equivalence classes of representations where the holonomy around the punctures $x_j$ has
been fixed to be of type $\cC_j$, $j=1,\ldots, s$.
If $\cC_j=[D_j]$, that is $D_j\in \cC_j$, then we may write
$\cR_{D_1,\ldots, D_s} (X,G)$ instead of $\cR_{\cC_1,\ldots, \cC_s} (X,G)$.

The topology and geometry of $G$-character varieties has been studied
extensively in the last two decades, starting with the foundational work
\cite{Hi}. Recently interest has been given to the algebro-geometric structure
of $G$-character varieties mainly because of the implications to Mirror Symmetry \cite{Ha,HaTh}. 
For this, computations of the Hodge-Deligne polynomials have
been done for a number of $G$-character varieties, mainly using arithmetic and combinatorial 
techniques \cite{HaRV,HaLeRV,Me}. In \cite{LoMuNe} a geometric method has been introduced to 
compute Hodge-Deligne polynomials of character varieties by analysing the spaces of matrices
\begin{align*}
  Z_{\cC_1,\ldots, \cC_s} & (X,G)  = \\
 =& \left\{(A_1,B_1,\ldots, A_g,B_g,C_1,\ldots, C_s)\in G^{2g+s}\,|\prod_{i=1}^g [A_i,B_i] \prod_{j=1}^s C_j=\Id, 
 C_j\in \cC_j, 1\leq j \leq s  \right\}
\end{align*}
explicitly (using geometric decompositions of the spaces). The technique is based on the work \cite{Mu} 
of the second author. In \cite{LoMuNe}, we dealt with the case $G=\SL(2,\CC)$, $s=1$ and $g=1,2$. 
In \cite{MaMu} the case $s=1$, $g=3$ is done. We address the case $s=2$, $g=1$ in this paper.
This moduli space is of interest since, when the conjugacy classes are diagonal and the 
eigenvalues have modulus different from one (see \cite[Theorem 0.4]{biquard-jardim}), it appears as
the moduli space of doubly periodic instantons. Indeed, 
the $\SL(2,\CC)$-character variety with diagonal and different conjugacy classes around the marked points is 
diffeomorphic to the moduli space of stable parabolic Higgs bundles of parabolic degree $0$ and traceless Higgs field.
This last moduli space, when the residue of the Higgs field is not nilpotent, is isomorphic, with the same complex structure, 
to the moduli space of doubly periodic instantons 
through the Nahm transform \cite{biquard-jardim, jardim}. Actually, the moduli space is hyperk\"ahler and both complex structures
-- the one given as $\SL(2,\CC)$-character variety, and the one given as moduli space of parabolic Higgs bundles 
(nilpotent or non-nilpotent) -- are
two of the complex structures in the family.

When  the conjugacy classes are diagonal and the eigenvalues  have modulus equal to one, the character variety is diffeomorphic 
to the moduli space of parabolic Higgs bundles, with nilpotent Higgs field, for which the Betti numbers are known \cite{BY}. 
As the character varieties are diffeomorphic for different values of the eigenvalues (having modulus equal to one is not relevant),
the Betti numbers are the same for the moduli space of double periodic instantons. So our results in 
Section \ref{sec:hodge} provide some Hodge numbers for the character variety diffeomorphic to the moduli space of doubly periodic instantons.

The case of several marked points has also been studied because of its relation to parabolic bundles
\cite{GPGoMu}. Here we want to address the first case of the computation of Hodge-Deligne polynomials
of character varieties for several marked points, namely the case of $G=\SL(2,\CC)$, $g=1$ and $s=2$.
For $G=\SL(2,\CC)$, 
there are five types of conjugacy classes, determined by the elements
$\Id$, $-\Id$, $J_+=\left( \begin{array}{cc} 1 & 1 \\ 0 & 1 \end{array} \right)$, 
$J_-=\left( \begin{array}{cc} -1 & 1 \\ 0 & -1 \end{array} \right)$, and
the diagonal matrices $\xi_\lambda =\left( \begin{array}{cc} \lambda & 0 \\ 0 & \lambda^{-1} 
\end{array} \right)$, for $\lambda \in \CC-\{0,\pm 1\}$. The last type is determined by
$\lambda$ up to $\lambda\sim \lambda^{-1}$.

Therefore there are $25$ possible character varieties 
 $$
  \cR_{\cC_1,\cC_2}(X, \SL(2,\CC)).
 $$
The symmetry
between $\cC_1, \cC_2$ reduces the number of cases to $15$ (see Lemma \ref{lem:symmetry}). Moreover, the cases where $\cC_1=[\Id]$ or
$\cC_2=[-\Id]$, correspond basically to the case of a one puncture elliptic curve,
computed in \cite{LoMuNe}. This means that there are $6$ cases left.
Our main result is the computation of the Hodge-Deligne polynomials of these character varieties
(see Section \ref{sec:2} for the definition of the Hodge-Deligne polynomial of an algebraic variety). 
These are as follows.

\begin{thm} \label{thm:main-result}
  We have the following:
 \begin{enumerate}
  \item $e(\cR_{J_+,J_+}(X, \SL(2,\CC))) =e(\cR_{J_-,J_-}(X, \SL(2,\CC))) =q^4+q^3-q+7$.
  \item $e(\cR_{J_+,J_-}(X, \SL(2,\CC))) = q^4-3q^2-6q$.
  \item $e(\cR_{J_+,\xi_\lambda}(X, \SL(2,\CC))) = e(\cR_{J_-,\xi_\lambda}(X, \SL(2,\CC)))= q^4 + q^3 + q^2 - 3q$. %q^4+q^3+2q^2+q+1$.
  \item $e(\cR_{\xi_\lambda,\xi_\mu}(X, \SL(2,\CC))) =q^4+2q^3+6q^2+2q+1$, for $\mu \neq \lambda^{\pm 1}$.
  \item $e(\cR_{\xi_\lambda,\xi_\lambda}(X, \SL(2,\CC))) =q^4+q^3+8q^2+q+1$.
 \end{enumerate}
The moduli spaces $\cR_{\cC_1,\cC_2}(X, \SL(2,\CC))$ contain reducibles in cases (1) and (5).
\end{thm} 

We recall that in a GIT quotient, the reducibles are the points with non-trivial stabilizers.
These produce lower-dimensional orbirts. All orbits which contain such lower-dimensional 
orbits in their closure must be identified in the GIT quotient (these are the semistable points, 
and the identification is commonly known as S-equivalence).

\noindent \textbf{Acknowledgements.} We thank Nigel Hitchin for asking us the question addressed in
the current paper during 
the meeting ``Nigel Hitchin LAB Retreat: Topology of moduli spaces of representations'', held in
Miraflores de la Sierra (Madrid) in March 2013. We want to acknowledge the kind invitation of
the organizers to participate, and the friendly and stimulating atmosphere during the meeting.
We thank the referee for a careful reading of the manuscript and for helpful comments
which have improved the exposition.

%%%%%%%%%%%%%%%%%%%%%%%%%%%%%%%%%%%%%%%%%%%%%%%%%%%%%%%
\section{Hodge-Deligne polynomials} \label{sec:2}
%%%%%%%%%%%%%%%%%%%%%%%%%%%%%%%%%%%%%%%%%%%%%%%%%%%%%%%

Our main goal is to compute the Hodge-Deligne polynomial of the $\SL(2,\CC)$-character variety 
of an elliptic curve with two marked points. We will follow the methods in \cite{LoMuNe}, 
so we collect some basic results from \cite{LoMuNe} in this section.

We start by reviewing the definition of the Hodge-Deligne polynomial. 
A pure Hodge structure of weight $k$ consists of a finite dimensional complex vector space
$H$ with a real structure, and a decomposition $H=\bigoplus_{k=p+q} H^{p,q}$
such that $H^{q,p}=\overline{H^{p,q}}$, the bar meaning complex conjugation on $H$.
A Hodge structure of weight $k$ gives rise to the so-called Hodge filtration, which is a descending filtration
$F^{p}=\bigoplus_{s\ge p}H^{s,k-s}$. We define $\Gr^{p}_{F}(H):=F^{p}/ F^{p+1}=H^{p,k-p}$.

A mixed Hodge structure consists of a finite dimensional complex vector space $H$ with a real structure,
an ascending (weight) filtration $\ldots \subset W_{k-1}\subset W_k \subset \ldots \subset H$
(defined over $\RR$) and a descending (Hodge) filtration $F$ such that $F$ induces a pure Hodge structure
of weight $k$ on each $\Gr^{W}_{k}(H)=W_{k}/W_{k-1}$. We define
 $$
 H^{p,q}:= \Gr^{p}_{F}\Gr^{W}_{p+q}(H)
 $$
and write $h^{p,q}$ for the {\em Hodge number} $h^{p,q} :=\dim H^{p,q}$.

Let $Z$ be any quasi-projective algebraic variety (maybe non-smooth or non-compact). 
The cohomology groups $H^k(Z)$ and the cohomology groups with compact support  
$H^k_c(Z)$ are endowed with mixed Hodge structures \cite{De}. 
We define the {\em Hodge numbers} of $Z$ by
 \begin{eqnarray*}
 h^{k,p,q}(Z)&=&h^{p,q}(H^k(Z))=\dim \Gr^{p}_{F}\Gr^{W}_{p+q}H^{k}(Z) ,\\
 h^{k,p,q}_{c}(Z)&=&h^{p,q}(H_{c}^k(Z))=\dim \Gr^{p}_{F}\Gr^{W}_{p+q}H^{k}_{c}(Z) .
 \end{eqnarray*}

The Hodge-Deligne polynomial, or $E$-polynomial is defined as 
 $$
 e(Z)=e(Z)(u,v):=\sum _{p,q,k} (-1)^{k}h^{k,p,q}_{c}(Z) u^{p}v^{q}.
 $$

When $h_c^{k,p,q}=0$ for $p\neq q$, the polynomial $e(Z)$ depends only on the product $uv$.
This will happen in all the cases that we shall investigate here. In this situation, it is
conventional to use the variable $q=uv$. If this happens, we say that the variety is {\it of balanced type}.
For instance, $e(\CC^n)=q^n$.

The key property of Hodge-Deligne polynomials that permits their calculation is that they are additive for
stratifications of $Z$. If $Z$ is a complex algebraic variety and
$Z=\bigsqcup_{i=1}^{n}Z_{i}$, where all $Z_i$ are locally closed in $Z$, then $e(Z)=\sum_{i=1}^{n}e(Z_{i})$,

\begin{prop}[Proposition 2.4 in \cite{LoMuNe}]\label{prop:basics2}
Suppose that $B$ is connected and $\pi:Z\to B$ is an algebraic fibre bundle with fibre $F$ 
(not necessarily locally trivial in the Zariski topology) and that the action of $\pi_1(B)$ on $H_c^*(F)$ is trivial. 
Supposse that $Z,F,B$ are smooth. Then $e(Z)=e(F)e(B)$.
\end{prop}

The hypotheses of Proposition \ref{prop:basics2} hold in particular in the following cases:
\begin{itemize}
\item $B$ is irreducible and $\pi$ is locally trivial in the Zariski topology. 
\item $\pi$ is a principal $G$-bundle with $G$ a connected algebraic group. 
%\item $Z$ is a $G$-space with isotropy $H<G$ such that $G/H\to Z\to B$ is a fiber bundle, 
%and $G$ is a connected algebraic group.
\end{itemize}
We shall use the above in the following form. Suppose that $Z$ is a space with a
free action of an algebraic group $G$, 
$H\subset G$ is a connected subgroup and $\overline{Z}{}\subset Z$ is
a subset such that $G\overline{Z}{}=Z$ and 
 \begin{equation}\label{eqn:referee}
 Hz_0=G z_0\cap \overline{Z}{}, \ \text{for any} \ z_0\in \overline{Z}{}.
 \end{equation}
Then, in particular $Z/G=\overline{Z}{}/H$. In this case we have an $H$-bundle $G\x \overline{Z}{} \to Z$.
Applying the above, $e(Z)=e(G\x \overline{Z}{})/e(H)$. Hence we can write
 \begin{equation}\label{eqn:Z}
 e(Z)=e(\overline{Z}{}) e(G/H).
 \end{equation}

We need to recall some Hodge-Deligne polynomials from \cite{LoMuNe}. First, 
we have that $e(\SL(2,\CC))=q^3-q$ and $e(\PGL(2,\CC))=q^3-q$.
Consider the following subsets of $\SL(2,\CC)$:
\begin{itemize}
\item $W_{0}:=$ conjugacy class of $\left(
             \begin{array}{cc}
               1 & 0 \\
               0 & 1
             \end{array}
           \right)$. It has $e(W_0)=1$.
\item $W_{1}:=$ conjugacy class of  $\left(
             \begin{array}{cc}
               -1 & 0 \\
               0 & -1
             \end{array}
           \right)$. It has $e(W_1)=1$.
\item $W_{2}:=$ conjugacy class of $J_+=\left(
              \begin{array}{cc}
                1 & 1 \\
                0 & 1
              \end{array}
            \right)$. It is $W_2 \cong \PGL(2,\CC)/U$, with $U=\left\{ \left(
              \begin{array}{cc}
                1 & y \\
                0 & 1
              \end{array} \right) \,|\, y\in\CC \right\}$. It has $e(W_2)=q^2-1$.
\item $W_{3}:=$ conjugacy class of $J_-=\left(
              \begin{array}{cc}
                -1 & 1 \\
                0 & -1
              \end{array}
            \right)$. It is $W_3\cong \PGL(2,\CC)/U$ and  $e(W_3)=q^2-1$.
\item $W_{4,\lambda}:=$ conjugacy class of
  $\xi_\lambda=\left(
              \begin{array}{cc}
                \lambda & 0 \\
                0 & \lambda^{-1}
              \end{array}
            \right)$, where $\lambda\in\CC-\{0,\pm1\}$. Note that $W_{4,\lambda}=W_{4,\lambda^{-1}}$,
since the matrices $\xi_\lambda$ and $\xi_{\lambda^{-1}}$ are conjugated.
We have $W_{4,\lambda} \cong \PGL(2,\CC)/D$, where $D=\left\{ \left(
              \begin{array}{cc}
                x & 0 \\
                0 & x^{-1}
              \end{array} \right) \,|\, x\in\CC^* \right\}$. So $e(W_{4,\lambda})=q^2+q$.
\item We also need the set $W_{4}:=\{A\in\SL(2,\CC)\, | \, \Tr(A)\ne\pm2\}$,
which is the union of the conjugacy classes $W_{4,\lambda}$, $\lambda \in \CC- \{0,\pm 1\}$.
This has $e(W_4)=e(\SL(2,\CC))-e(W_0)-e(W_1)-e(W_2)-e(W_3)= q^3-2q^2-q$.
\end{itemize}

Now consider the map
\begin{eqnarray*}
f:\SL(2,\CC)^{2}&\longrightarrow & \SL(2,\CC)\\
(A,B)&\mapsto& [A,B]=ABA^{-1}B^{-1}
\end{eqnarray*}
Note that $f$ is equivariant under the action of $\SL(2,\CC)$ by conjugation on both spaces. 
We stratify $X=\SL(2,\CC)^{2}$ as follows
\begin{itemize}
\item $X_{0}:=f^{-1}(W_{0})$,
\item $X_{1}:=f^{-1}(W_{1})$,
\item $X_{2}:=f^{-1}(W_{2})$,
\item $X_{3}:=f^{-1}(W_{3})$,
\item $X_{4}:=f^{-1}(W_{4})$.
\end{itemize}
We also introduce the varieties $f^{-1}(C)$ for fixed $C\in \SL(2,\CC)$ and define accordingly
 \begin{itemize}
\item $\overline{X}_{0}:=f^{-1}(\Id)=X_0$,
\item $\overline{X}_{1}:=f^{-1}(-\Id)=X_1$,
\item $\overline{X}_{2}:=f^{-1}(J_+)$. Then there is a fibration $U \to \PGL(2,\CC) \x \overline{X}_2 \to X_2$,
and by (\ref{eqn:Z}), $e(X_2)=(q^2-1)e(\overline{X}_2)$. Note that the action of $\PGL(2,\CC)$ on $X_2$ is
free because there are no reducibles (see \cite{LoMuNe}).
\item $\overline{X}_{3}:=f^{-1}(J_-)$. Again there is a fibration $U \to \PGL(2,\CC) \x \overline{X}_3\to X_3$,
and $e(X_3)=(q^2-1)e(\overline{X}_3)$.
\item $\overline{X}_{4,\lambda}:=f^{-1}(\xi_\lambda)$, for $\lambda\ne 0,\pm1$.
We define also $X_{4,\lambda}=f^{-1}(W_{4,\lambda})$. There is a fibration 
$D\to \PGL(2,\CC) \x\overline{X}_{4,\lambda} \to X_{4,\lambda}$,
and $e(X_{4,\lambda})=(q^2+q)e(\overline{X}_{4,\lambda})$.
\end{itemize}

It will also be convenient to define
\begin{itemize}
\item $\overline{X}_{4}:=\left\{(A,B,\lambda)\, | \, [A,B]=\left(\begin{array}{cc}
      \lambda & 0 \\
      0 & \lambda^{-1}\\
    \end{array}
  \right), \lambda\ne0, \pm1,\, A,\,B \in \SL(2,\CC)\right\}$.
\end{itemize}
There is an action of $\ZZ_2$ on $\overline{X}_4$ given by interchanging 
$(A,B,\lambda)\mapsto (P^{-1}AP, P^{-1}BP,\lambda^{-1})$, with $P=
\left(\begin{array}{cc}
      0 & 1 \\
      1 & 0\\
    \end{array} \right)$. The map $\overline{X}_4 \to \overline{X}_4/\ZZ_2$ is
equivalent to the map $(A,B,\lambda) \to (A,B)$, so 
  $$
 \overline{X}_{4}/\ZZ_2 = \left\{(A,B)\, | \, [A,B]=\left(\begin{array}{cc}
      \lambda & 0 \\
      0 & \lambda^{-1}\\
    \end{array}
  \right), \lambda\ne0, \pm1,\, A,\,B \in \SL(2,\CC)\right\}
 $$
coincides with the union of all $\overline{X}_{4,\lambda}$.

The Hodge-Deligne polynomials computed in \cite{LoMuNe} are as follows:
\begin{align*}
 e(X_{0}) &=  q^{4}+4q^{3}-q^2-4q \\
 e(X_1) &=e( \PGL(2,\CC))=q^{3}-q \\
 e(\overline{X}_2) &= q\big( (q-1)^2-4 \big) = q^3-2q^2-3 q \\
 e(X_{2}) &= (q^2-1) (q^3-2q^2-3q) =q^{5}-2q^{4}-4q^{3}+2q^{2}+3q \\
 e(\overline{X}_3) & =q(q^2+3q)=q^3 +3 q^2 \\
 e(X_{3}) & = (q^2-1)(q^3+3q^2)=q^5  + 3 q^4 - q^3 -3 q^2 \\
 e(\overline{X}_{4,\lambda}) &=(q-1)(q^2+4q+1)=q^3 + 3 q^2 - 3 q -1 \\
 e({X}_{4,\lambda}) &= (q^2+q)(q^3+3q^2-3q-1)=q^5+4q^4-4q^2-q \\
 e(\overline{X}_4) &= q^4-3q^3-6q^2+5q+3 \\
 e(\overline{X}_4/\ZZ_2) &= q^4-2q^3-3q^2+3q+1 \\
 e(X_{4}) & =q^6 - 2 q^5 - 4 q^4 + 3 q^2 +2 q 
\end{align*}

We will call holonomies of Jordan type those which belong to one of the conjugacy classes $J_{+}$ and 
$J_{-}$, and of diagonalisable type those which belong to $ \Id$, $-\Id$ or $\xi_{\lambda}$.

%%%%%%%%%%%%%%%%%%%%%%%%%%%%%%%%%%%%%%%%%%%%%%%%%%%%%%%
\section{Holonomies of Jordan type}
%%%%%%%%%%%%%%%%%%%%%%%%%%%%%%%%%%%%%%%%%%%%%%%%%%%%%%%

Let $\cC_1,\cC_2$ be conjugacy classes in $\SL(2,\CC)$. We want to study the set
$$
 Z(\cC_1,\cC_2)=\{ (A,B,C_1,C_2) \in \SL(2, \CC) \, | \,  [A,B]C_1C_2=\Id, C_1\in \cC_1, C_2\in \cC_2\}
$$
and 
 $$
\cR_{\cC_1, \cC_2}=Z(\cC_1,\cC_2) // \PGL(2,\CC).
 $$
 
Recall that there are five types of conjugacy classes in $\SL(2,\CC)$, namely
$\Id$, $-\Id$, $J_+=\left(\begin{array}{cc}1& 1  \\ 0 & 1\end{array}\right)$,
$J_-=\left(\begin{array}{cc}-1& 1  \\ 0 & -1\end{array}\right)$,
and $\xi_\lambda=\left(\begin{array}{cc}\lambda & 0  \\ 0 & \lambda^{-1}
\end{array}\right)$, for $\lambda\neq 0,\pm 1$, and defined up to $\lambda \sim \lambda^{-1}$. 
Therefore there is a total of $25$ possible combinations for $\cC_1,\cC_2$. The following symmetry 
reduces the number of cases.

\begin{lem} \label{lem:symmetry}
$Z(\cC_1,\cC_2) \cong  Z(\cC_2,\cC_1)$.
\end{lem}

\begin{proof}
 The equation $[A,B]C_1C_2=\Id$ is equivalent to $[A,B]=C_2^{-1}C_1^{-1}$. This can be inverted to
give $[B^{-1},A^{-1}]=C_1C_2$, i.e. $[B^{-1},A^{-1}]C_2^{-1}C_1^{-1}=\Id$.
So the map $(A,B,C_1,C_2)\mapsto (B^{-1},A^{-1}, C_2^{-1},C_1^{-1})$ gives the required isomorphism.
Note that when $C\in \cC$, $C^{-1}$ runs over $\cC$ again. 
\end{proof}

If $\cC_1=W_0=\{\Id\}$, then 
\begin{itemize}
 \item $Z_{0j}=Z(W_0,W_j)=X_j$, $j=0,1,2,3$
 \item $Z_{04}^\lambda=Z(W_0,W_{4,\lambda})=X_{4,\lambda}$
 \item $Z_{04}=\bigcup Z_{04}^\lambda= X_4$
\end{itemize}

If $\cC_1=W_1=\{-\Id\}$, then 
\begin{itemize}
 \item $Z_{11}=Z(W_1,W_1)=X_0$
 \item $Z_{12}=Z(W_1,W_2)=X_3$
 \item $Z_{13}=Z(W_1,W_3)=X_2$
 \item $Z_{14}^\lambda=Z(W_0,W_{4,\lambda})=X_{4,-\lambda}$
 \item $Z_{14}=\bigcup Z_{14}^\lambda=X_4$
\end{itemize}

The Hodge-Deligne polynomials have been computed in \cite{LoMuNe}. They are 
\begin{itemize}
 \item $e(Z_{00})=e(Z_{11})=q^4+4q^3-q^2-4q$
 \item $e(Z_{01})=q^3-q$
 \item $e(Z_{02})=e(Z_{13})=q^5-2q^4-4q^3+2q^2+3q$
 \item $e(Z_{03})=e(Z_{12})=q^5+3q^4-q^3-3q^2$
 \item $e(Z_{04}^\lambda)=e(Z_{14}^\lambda)=q^5+4q^4-4q^2-q$
% \item $e(Z_{04})=e(Z_{14})=q^6-2q^5-4q^4+3q^2+2q$
\end{itemize}
and the Hodge-Deligne polynomials of the character varieties are \cite[Theorem 1.1]{LoMuNe}
\begin{itemize}
 \item $e(\cR_{\Id,\Id})=e(\cR_{-\Id,-\Id}) =q^2+1$
 \item $e(\cR_{\Id,-\Id}) =1$
 \item $e(\cR_{\Id,J_+})=e(\cR_{-\Id,J_-}) = q^2-2q+3$
 \item $e(\cR_{\Id,J_-})=e(\cR_{-\Id,J_+}) = q^2+3q$
 \item $e(\cR_{\Id,\xi_\lambda})=e(\cR_{-\Id,\xi_\lambda}) = q^2+4q+1$
\end{itemize}

Now we move to the cases when the holonomy around the punctures is given by the Jordan forms.

\subsection{The case $\cC_1=W_2=[J_+]$, $\cC_2=W_2=[J_+]$} \label{subsec:first}
Let
\begin{align*}
 Z(W_2,W_2) &=Z_{22}= \{(A,B,C_1,C_2) \, | \, C_1,C_2\in W_2, [A,B]C_1=C_2^{-1}\},  \\
 \overline{Z}{}(W_2,W_2) &=\overline{Z}{}_{22}= \{(A,B,C)\, | \, C\in W_2, [A,B]C=J_+\} .
 \end{align*}
The action of $\PGL(2,\CC)$ on $Z_{22}$ is free except when a non-trivial element $P$ fixes
simultaneously $A,B,C_1,C_2$. Write $C_2=QJ_+Q^{-1}$, for some $Q\in \PGL(2,\CC)/U$.
Then $P\in QUQ^{-1}$, and hence $A,B \in Q (U\cup (-U)) Q^{-1}$. So $[A,B]=\Id$ and $C_1=C_2^{-1}$.
Analogously, the action of $U$ on points of $\overline{Z}{}_{22}$
is free except when $A,B \in U\cup (-U), C=J_+$. The set of 
reducibles is thus
\begin{align*}
  \cD &=\{(A,B,C)| A,B\in Q (U\cup (-U)) Q^{-1}, C_1^{-1}=C_2 =QJ_+Q^{-1} \in W_2\} \subset Z_{22} , \\
  \overline\cD &=\{(A,B,C)| A,B\in U\cup (-U), C=J_+\}  \subset \overline{Z}{}_{22}.
 \end{align*}
Denote by $Z_{22}^*=Z_{22}-\cD$ and $\overline{Z}{}_{22}^*=\overline{Z}{}_{22}-\overline\cD$
the set of irreducible representations. Then $\PGL(2, \CC)$ acts freely on $Z_{22}^*$, 
clearly $\PGL(2,\CC) \, \overline{Z}{}_{22}={Z}{}_{22}$, and also  Condition \eqref{eqn:referee} 
holds: if $(A,B,C)\in \overline{Z}{}_{22}$ and $P\in \PGL(2,\CC)$ satisfies that
 $(PAP^{-1},PBP^{-1},PCP^{-1})\in \overline{Z}{}_{22}$, then $P$ fixes $J_+$, so $P\in U$. Therefore
there is a fibration 
 $$
U  \to \PGL(2,\CC) \x \overline{Z}{}_{22}^* \to Z_{22}^*,
$$ 
so $e(Z_{22}^*)=(q^2-1) e(\overline{Z}{}_{22}^*)$.

Let $(A,B,C)\in \overline{Z}{}_{22}$.
As $C\sim J_+$, there is some $P\in \SL(2,\CC)$, well-defined up to sign and up to multiplication by $U$
on the right, such that $C=P J_+ P^{-1}$. The equation $[A,B]P  J_+ P^{-1}= J_+$ is rewritten
as $[A,B][P, J_+]=\Id$. So 
$$
 \overline{Z}{}_{22} = \{(A,B,P)\in \SL(2,\CC)^3  | \, [A,B][P, J_+]=\Id\} / \ZZ_2 \x U ,
 $$
where $\ZZ_2$ acts by $P\mapsto -P$,
and $U$ acts by multiplication on the right on $P$. This is a free action.

Write $P=\left( \begin{array}{cc}x& y  \\ z & w\end{array}\right)$,
then we have 
 \begin{equation}\label{eqn:are}
 [P, J_+]=\left( \begin{array}{cc}1-xz & -1+x(x+z)  \\ -z^2 & 1+(x+z)z \end{array}\right).
 \end{equation}
Note that the trace is
 $$
t=\Tr [P, J_+]=2+z^2\, .
 $$
The action of $U$ on $P$ moves $y$ and $w$, so the class of $P$ modulo $\ZZ_2\x U$
is determined by $(x,z)\in \CC^2-\{0\}$, up to sign. 
Hence 
 $$
 \overline{Z}{}_{22} \cong \{(A,B,(x,z))\in \SL(2,\CC)^2 \x ((\CC^2-\{0\})/\pm) \,  | \, [A,B]^{-1}=[P, J_+]\} .
 $$
The space  $\overline{Z}{}_{22}$ is stratified as follows:
\begin{enumerate}
\item If $z=0$ then $x\neq 0$. So $P=\left( \begin{array}{cc}x& y  \\ 0 & x^{-1} \end{array}\right)$.
For $x\neq \pm 1$, we have $[P,J_+]\sim J_+$ so $(A,B)\in \overline X_2$.
Hence the contribution to the Hodge-Deligne polynomial is
$e((\CC^*-\{\pm 1\})/\pm) e(\overline X_2)=(q-2) e(\overline X_2)$.
\item $z=0$ and $x=\pm 1$. As it is defined up to sign, we can arrange $x=1$. 
So $P=\Id$, and $(A,B)\in X_0$. The contribution is $e(X_0)$.
\item $z=\pm 2i$. As it is defined up to sign, can choose $z=2i$. Then $[P,J_+]\sim J_-$, and hence 
the contribution is $q\, e(\overline X_3)$.
\item $z\neq 0,\pm 2i$. Now we have a fibration $(A,B) \mapsto t=2+z^2$, where $z$ is defined up to sign,
so $z^2\in \CC-\{0, -4\}$. Take $v=xz$, $v\in \CC$. So
$[P, J_+]=\left( \begin{array}{cc} 1-v & * \\ -z^2 & 1+v+z^2 \end{array}\right)$. Hence the
Hodge-Deligne polynomial is $e(\CC) e(\overline X_4/\ZZ_2)$.
\end{enumerate}

Putting all together, and using $e(\overline\cD)=4q^2$, we have
 \begin{align*}
 e(\overline{Z}{}_{22}) &= (q-2) e(\overline X_2) + q e(\overline X_3)+e(X_0) + q e(\overline X_4/\ZZ_2) \\
 &= q^5+q^4+3q^2+3q, \\
e(\overline{Z}{}_{22}^*) &=e(\overline{Z}{}_{22}) - e(\overline\cD) =q^5+q^4- q^2+3q \\
 e(Z_{22}^*) &= (q^2-1) e(\overline{Z}{}_{22}^*) = q^7+q^6-q^5-2q^4+3q^3+q^2-3q.
\end{align*}

Finally, we compute the Hodge-Deligne polynomial of $\cR_{W_2,W_2}=Z_{22}//\PGL(2,\CC)$.
It is clear that $Z_{22}^*/\PGL(2,\CC)=\overline{Z}{}_{22}^*/U$. The contribution of the
reducibles is as follows. First note that $\cD/\PGL(2,\CC)=\overline\cD/U$. Now
let $(A,B,C)\in \overline\cD$. Then $A,B\in U\cup (-U), C=J_+$.
If $A=\left(\begin{array}{cc} 1 & a\\ 0 & 1\end{array}\right)$,
$B=\left(\begin{array}{cc} 1 & b\\ 0 & 1\end{array}\right)$, we consider
$A'=\left(\begin{array}{cc} x & a\\ 0 & x^{-1}\end{array}\right)$,
$B'=\left(\begin{array}{cc} y & b\\ 0 & y^{-1}\end{array}\right)$, with $(x-x^{-1})b =(y-y^{-1})a=:\eta ab$,
so $[A',B']=[A,B]$. 
When $x,y\to 1$, we have $A'\to A, B'\to B$.
The action of $\left(\begin{array}{cc} 1 & \alpha \\ 0 & 1\end{array}\right)$ on $A',B'$ takes
$(a,b) \mapsto (a+\alpha(x-x^{-1}), b + \alpha(y-y^{-1}))=(a+\alpha\eta a, b + \alpha\eta b)$.
So going to the limit $x,y\to 1$, $(a,b) \sim (a+\alpha\eta a, b + \alpha\eta b)$.
For $\alpha=-\eta^{-1}$, they converge to $(\Id, \Id)$. Taking into account the possible signs,
this means that the contribution of $\overline\cD$ in the quotient consists of $4$ points.

So the contribution is
 \begin{align*}
 e(\cR_{W_2,W_2}) &= e(\overline{Z}{}_{22}^*)/e(U) + 4 =q^4+q^3-q+7.
 \end{align*}

\subsection{The case $\cC_1=W_3=[J_-]$, $\cC_2=W_3=[J_-]$}\label{subsec:referee},
There is an isomorphism  
 \begin{align*}
   Z(\cC_1,\cC_2) &= \{(A,B,C_1,C_2)\, | \, C_1 \in \cC_1,C_2\in \cC_2, [A,B]C_1C_2=\Id\}\\
    &= \{(A,B,C_1,C_2)\, | \, C_1 \in \cC_1,C_2\in \cC_2, [A,B](-C_1)(-C_2)=\Id\}\\
  &= Z(-\cC_1,-\cC_2).
 \end{align*}
In particular
 $$
 Z(W_3,W_3)=Z(W_2,W_2),
 $$
using that $W_3=[J_-]=[-J_+]=-W_2$.

Therefore
 \begin{align*}
% e(\overline{Z}{}_{33}) &= (q-2) e(\overline X_2)+e(X_0) + q e(\overline X_3)+ q e(\overline X_4/\ZZ_2) \\
% &= q^5+q^4+3q^2+3q, \\
% e(Z_{33}) &= (q^2-1) e(\overline{Z}{}_{33}) = q^7+q^6-q^5-q^4+3q^4+3q^3-3q^2-3q, \\
 e(\cR_{W_3,W_3}) & =q^4+q^3-q+7.
 \end{align*}

\subsection{The case $\cC_1=W_2=[J_+]$, $\cC_2=W_3=[J_-]$}\label{subsec:third}
 Now we choose $W_2=[J_+]$ but $W_3=[-J_+]$, using that $-J_+ \sim J_-$. So let
 \begin{align*}
  Z({W_2,W_3}) &= Z_{23}= \{(A,B,C_1,C_2)\, | \, C_1\in [J_+] ,C_2\in [J_+], [A,B]C_1=-C_2^{-1}\}, \\
 \overline{Z}{}({W_2,W_3}) &= \overline{Z}{}_{23}= \{(A,B,C) \, | C\in W_2, \, [A,B]C=-J_+\}.
 \end{align*}
So 
 $$
 \overline{Z}{}({W_2,W_3}) = \overline{Z}{}_{23}= \{(A,B,P) \, | \, [A,B][P, J_+]=-\Id\}/\ZZ_2\x U.
 $$
Note that  $\PGL(2, \CC)$ acts freely on $Z_{23}$, since if a non-trivial $P$ fixes $(A,B,C_1,C_2)$ then 
$[A,B]=\Id$, and this cannot be possible. Also  Condition \eqref{eqn:referee} holds here, which is proved 
as in Subsection \ref{subsec:first}. This means that 
there is a fibration $U  \to \PGL(2,\CC) \x \overline{Z}{}_{23} \to Z_{23}$, so
$e(Z_{23})= (q^2-1) e(\overline{Z}{}_{23})$.  

%Let $[A,B]=\left( \begin{array}{cc}a& b  \\ c & d\end{array}\right)$,
%so $[P, J_+]=\left( \begin{array}{cc}d& -b  \\ -c & a\end{array}\right)$.
Writing $P=\left( \begin{array}{cc}x& y  \\ z & w\end{array}\right)$,
we have 
 $$
 [P, J_+]=\left( \begin{array}{cc}1-xz & -1+x(x+z)  \\ -z^2 & 1+(x+z)z \end{array}\right),
 $$
and the trace is
 $$
t=\Tr [P, J_+]=2+z^2.
 $$

The space $\overline{Z}{}_{23}$ is stratified as follows:
\begin{enumerate}
\item If $z=0$ and $x\neq \pm 1$, then $[P,J_+]\sim J_+$ so $[A,B]\sim J_-$ and $(A,B)\in \overline X_3$.
As $x$ is defined up to sign, we get the contribution
$e((\CC^*-\{\pm 1\})/\pm) e(\overline X_3)=(q-2)e(\overline X_3)$. 
\item $z=0$ and $x=\pm 1$. As it is defined up to sign, we can arrange $x=1$. So $[P,J_+]=\Id$ 
and $(A,B)\in X_1$. So we get the contribution $e(X_1)$.
\item If $z=\pm 2i$, as it is defined up to sign, can choose $z=2i$. 
Then $[P,J_+]\sim J_-$. As $x\in \CC$, we have the contribution $q\, e(\overline X_2)$.
 \item $z\neq 0,\pm 2i$. Now we have a fibration $(A,B) \mapsto t=2+z^2$, where $z$ is defined up to sign,
so $z^2\in \CC-\{0, -4\}$. Take $v=xz$, $v\in \CC$. So
$[P, J_+]=\left( \begin{array}{cc} 1-v & * \\ -z^2 & 1+v+z^2 \end{array}\right)$. Hence the
Hodge-Deligne polynomial is $e(\CC) e(\overline X_4/\ZZ_2)$.
\end{enumerate}

Putting all together, 
 \begin{align*}
 e(\overline{Z}{}_{23}) &= (q-2) e(\overline X_3)+e(X_1) + q e(\overline X_2)+ q e(\overline X_4/\ZZ_2) \\
 &= q^5-3q^3-6q^2, \\
 e(Z_{23}) &= (q^2-1) e(\overline{Z}{}_{23}) = q^7-4q^5-6q^4+3q^3+6q^2.
\end{align*}

Finally, we want to compute the Hodge-Deligne polynomial of $\cR_{W_2,W_3}=Z_{23}//\PGL(2,\CC)$.
In this case the action is free, and there are no reducibles. So 
 $$
 e(\cR_{W_2,W_3})= q^4-3q^2-6q.
 $$

%%%%%%%%%%%%%%%%%%%%%%%%%%%%%%%%%%%%%%%%%%%%%%%%%%%%%%%%%%%%%%%
\section{One holonomy of Jordan type and the other of diagonalisable type}
%%%%%%%%%%%%%%%%%%%%%%%%%%%%%%%%%%%%%%%%%%%%%%%%%%%%%%%%%%%%%%%

\subsection{The case $\cC_1=W_2$, $\cC_2=W_{4,\lambda}$}\label{subsec:previous}
Now we have
\begin{align*}
  Z(W_2, W_{4,\lambda}) &= Z_{24}^\lambda =  \{ (A,B,C_1,C_2) \,|\, C_1\in W_2, C_2\in W_{4,\lambda}, 
  [A,B]C_1=C_2^{-1}\}, \\
  \overline{Z}{}(W_2, W_{4,\lambda}) &= \overline{Z}{}_{24}^\lambda =  \{ (A,B,C) \,|\, C\in W_2, 
  [A,B]C= D \},
\end{align*}
where $D=\left( \begin{array}{cc}\lambda^{-1} & 0  \\ 0 & \lambda \end{array}\right)$. The action of
$\PGL(2,\CC)$ is free as in Subsection \ref{subsec:third}, and  Condition \eqref{eqn:referee} holds again, which is proved 
as in Subsection \ref{subsec:first}. Therefore
there is a fibration $\CC^*  \to \PGL(2,\CC) \x \overline{Z}{}_{24}^\lambda \to Z_{24}^\lambda$, and
$e(Z_{24}^\lambda)= (q^2+q) e(\overline{Z}{}_{24}^\lambda)$. 

Writing $C=P J_+P^{-1}$, we get
 $$
  \overline{Z}{}_{24}^\lambda \cong  \{ (A,B,P) \,|\, [A,B][P, J_+]= D J_+^{-1}\} / \ZZ_2\x U ,
 $$
where $D J_+^{-1}=\left( \begin{array}{cc}\lambda^{-1} & -\lambda^{-1}  \\ 0 & \lambda \end{array}\right)$.
%We have to quotient by $\ZZ_2$ changing the sign of $P$,
%by an action of $U=\CC$ on the right to $P$,
%and then by an action of $V=\CC^*$ (diagonal matrices) acting as
%$A,B,P\mapsto Q^{-1}AQ, Q^{-1}B Q. Q^{-1}P$.

Let $[A,B]=\left( \begin{array}{cc}a& b  \\ c & d\end{array}\right)$, so 
 \begin{equation}\label{eqn:18a}
 [P, J_+]=[A,B]^{-1}DJ_{+}^{-1}=
 \left( \begin{array}{cc}d& -b  \\ -c & a\end{array}\right)
\left( \begin{array}{cc}\lambda^{-1} & -\lambda^{-1}  \\ 0 & \lambda \end{array}\right) =
 \left( \begin{array}{cc}\lambda^{-1} d & -\lambda^{-1} d-\lambda b   \\ 
 -\lambda^{-1} c & \lambda^{-1} c + \lambda a \end{array}\right).
 \end{equation}
Writing $P=\left( \begin{array}{cc}x& y  \\ z & w\end{array}\right)$,
we have 
 \begin{equation}\label{eqn:18b}
 [P, J_+]=\left( \begin{array}{cc}1-xz & -1+x(x+z)  \\ -z^2 & 1+(x+z)z \end{array}\right).
 \end{equation}
Quotienting by the action of $U$ is equivalent to forgetting $y$ and $w$. Hence
the quotient is determined by $(x,z)\in \CC^2-\{(0,0)\}$, modulo sign.
Equating (\ref{eqn:18a}) and (\ref{eqn:18b}), we get
 \begin{align*}
 a &= \lambda^{-1} (1+xz) \\
 b &= -\lambda^{-1} x^2 \\
 c &= \lambda z^2 \\
 d &= \lambda(1-xz).
 \end{align*}

We have
$$
 t=\Tr ([A,B])=a+d= \lambda (1-xz) + \lambda^{-1} (1+xz) =\lambda+\lambda^{-1} - xz (\lambda-\lambda^{-1}).
$$
So we get the following strata for $\overline{Z}{}_{24}^\lambda$\, :
\begin{itemize}
\item For $xz=0$, $t=\lambda+\lambda^{-1}$. The contribution from the values $(x,z)$ is $2q-2$, so we
get the contribution $(2q-2) e(\overline X_{4,\lambda})$.
\item For $t=2$, we have $xz=c_0$ for some $c_0\neq 0$. Quotienting by the change of sign, $(x,z)$ move
in a $\CC^*$. Also $[A,B]\neq \Id$, so $[A,B]\sim J_+$. This gives the contribution $(q-1) e(\overline X_2)$.
\item For $t=-2$, the computation is similar and the contribution is $(q-1) e(\overline X_3)$.
\item For $t\neq \pm 2, \lambda+\lambda^{-1}$. We have a fibration with fiber $\CC^*$ parametrizing
the values of $(x,z)$, for fixed $xz=(\lambda+\lambda^{-1}-t)/(\lambda - \lambda^{-1})$. 
For each value of $t\in \CC-\{\pm 2, \lambda+\lambda^{-1}\}$, we have
$(A,B)\in \overline X_{4,\lambda}$. The union of all of them is $\overline X_4/\ZZ_2$. The total contribution is
thus $(q-1) \left(e(\overline X_4/\ZZ_2)- e(\overline X_{4,\lambda})\right)$.
\end{itemize}

Putting all together,
 \begin{align*}
 e(\overline{Z}{}_{24}^\lambda) &= (2q-2) e(\overline X_{4,\lambda})+ ( q-1) e(\overline X_2)+ (q-1) e(\overline X_3)+
(q-1) \left(e(\overline X_4/\ZZ_2)- e(\overline X_{4,\lambda})\right) \\
 &= q^5 - 4q^2 + 3q , \\ %q^5+q^3-q^2-1 ,\\
 e({Z}_{24}^\lambda) &= (q^2+q) e(\overline{Z}{}_{24}^\lambda)= q^7 + q^6 - 4q^4 - q^3 + 3q^2. %q^7+q^6+q^5-q^3-q^2-q.
 \end{align*}

To get $\cR_{W_2,W_{4,\lambda}}$, we need to quotient by the action of $\CC^*$, corresponding to the diagonal
matrices acting by conjugation on $(A,B,P)$. There are no fixed points, since in such case
$[A,B]=\Id$, which does not happen. 
%sends $(a,b,c,d,x,z) \mapsto (a, \omega^2 b, \omega^{-2}c, d, \omega x , \omega^{-1} z)$.
Hence
 $$
 e(\cR_{W_2,W_{4,\lambda}})= q^4 + q^3 + q^2 - 3q. %q^4+q^3+2q^2+q+1.
 $$

\subsection{The case $\cC_1=W_3$ and $\cC_2=W_{4,\lambda}$}
Note that $W_3=-W_2$ and $W_{4,\lambda}=-W_{4,-\lambda}$. Hence
the isomorphism $Z(\cC_1,\cC_2) = Z(-\cC_1,-\cC_2)$, described at the beginning
of Subsection \ref{subsec:referee}, gives here
 $$
 Z(W_3,W_{4,\lambda})= Z(W_2,W_{4,-\lambda})=Z{}_{24}^{-\lambda}.
 $$
Therefore using the results of Subsection \ref{subsec:previous}, we have
 \begin{align*}
 e(\overline{Z}{}_{34}^\lambda) &= q^5 - 4q^2 + 3q , \\ %q^5+q^3-q^2-1 ,\\ 
 e({Z}_{34}^\lambda) &=  q^7 + q^6 - 4q^4 - q^3 + 3q^2, \\ % q^7+q^6+q^5-q^3-q^2-q, \\
 e(\cR_{W_3,W_{4,\lambda}}) &= q^4 + q^3 + q^2 - 3q. % q^4+q^3+2q^2+q+1.
 \end{align*}

%%%%%%%%%%%%%%%%%%%%%%%%%%%%%%%%%%%%%%%%%%%%%%%%%%%%%%%%%%%%%%%%%%%
\section{Holonomies of diagonalizable type}\label{sec:diag}
%%%%%%%%%%%%%%%%%%%%%%%%%%%%%%%%%%%%%%%%%%%%%%%%%%%%%%%%%%%%%%%%%%%

Let $D_1=\left(\begin{array}{cc}\lambda_1 & 0\\ 0 & \lambda_1^{-1}\end{array}\right)$
and $D_2=\left(\begin{array}{cc}\lambda_2 & 0\\ 0 & \lambda_2^{-1}\end{array}\right)$,
with $\lambda_1,\lambda_2\neq 0,\pm 1$.
We want to understand the set
 $$
 Z_{44}^{\lambda_1\lambda_2}= \{(A,B,C_1,C_2) | C_1\in [D_1],C_2\in [D_2], [A,B]C_1=C_2^{-1} \}
 $$
and the quotient $\cR_{\xi_{\lambda_1}, \xi_{\lambda_2}}=Z_{44}^{\lambda_1\lambda_2}//\PGL(2,\CC)$. All
orbits have trivial stabilizers, except in one case: when $A,B,C_1,C_2$ are diagonal with
respect to the same basis, and hence $\lambda_1\lambda_2=1$, that is $\lambda_2=\lambda_1^{-1}$.
As $\lambda$ is defined up to $\lambda\sim \lambda^{-1}$, we also have reducibles in the
case $\lambda_2=\lambda_1$.

Suppose that $\lambda_1\neq \lambda_2,\lambda_2^{-1}$ from now on in this section. Let
 $$
  \overline{Z}{}_{44}^{\lambda_1\lambda_2}= \{(A,B,C) | C \in [D_1],  [A,B]C=D_2^{-1} \}.
 $$
As we said above, the action of $\PGL(2,\CC)$ is free on $ Z_{44}^{\lambda_1\lambda_2}$. 
We check this as follows: write $C_1=QD_1 Q^{-1}$. Then if a non-trivial $P$ fixes $A,B,C_1,C_2$, 
it must be that $P\in Q \CC^* Q^{-1}$, where $\CC^* \subset \PGL(2,\CC)$ denotes the
diagonal matrices. This forces that $A,B,C_2\in  Q \CC^* Q^{-1}$, hence $[A,B]=\Id$ and
$C_2=C_1^{-1}$, which contradicts our assumption. In the second place, it is clear
that $\PGL(2,\CC) \,  \overline{Z}{}_{44}^{\lambda_1\lambda_2}= {Z}{}_{44}^{\lambda_1\lambda_2}$.
Finally,  Condition \eqref{eqn:referee} holds as follows: let $P\in \PGL(2,\CC)$ such that 
  $(A,B,C), (PAP^{-1},PBP^{-1},PCP^{-1}) \in  \overline{Z}{}_{44}^{\lambda_1\lambda_2}$.
Then $P$ fixes $D_2^{-1}$, so $P \in \CC^*$. All together, we
have a fibration 
  $$
 \CC^*  \to \PGL(2,\CC) \x  \overline{Z}{}_{44}^{\lambda_1\lambda_2} \to {Z}{}_{44}^{\lambda_1\lambda_2}.
 $$
So $\cR_{\xi_{\lambda_1},\xi_{\lambda_2}}=\overline{Z}{}_{44}^{\lambda_1\lambda_2}//\CC^*$.

To describe $\overline{Z}{}_{44}^{\lambda_1\lambda_2}$, note that 
as $C\sim D_1$, there exists $P\in \SL(2,\CC)$ with $C=P D_1P^{-1}$. Such $P$ is
defined up to $\CC^*$, the diagonal matrices acting by multiplication on the right on $P$
(whcih is a free action). So we can rewrite
 $$
  \overline{Z}{}_{44}^{\lambda_1\lambda_2}= \{(A,B,P) | [A,B][P,D_1]=D_2^{-1}D_1^{-1} \}/\CC^*.
 $$
Set $D:=D_1=\left(\begin{array}{cc}\lambda& 0\\0&\lambda^{-1}\end{array}\right)$, where
$\lambda=\lambda_1$, and $\xi_{\mu}:=D_2^{-1}D_1^{-1}=\left(\begin{array}{cc}\mu &0\\0&\mu^{-1}\end{array}\right)$,
where $\mu=\lambda_1^{-1}\lambda_2^{-1} \neq 0,1$ and $\lambda^2\mu=\lambda_1^{-1}\lambda_2 \neq 0,1$.

The action of $Q\in \CC^\ast$ on the set of 
$(A,B,P)$ is by conjugation $A\mapsto Q^{-1}AQ$, $B\mapsto Q^{-1}BQ$, $PDP^{-1}\mapsto Q^{-1}PDP^{-1}Q$, or
equivalently, $P\mapsto Q^{-1}P$.

We have $\lambda_1=\lambda$, $\lambda_2=\lambda^{-1}\mu^{-1}$.
The case $\mu=-1$ corresponds to $\lambda_2=-\lambda^{-1}$. This case
is equivalent to $\lambda_2=-\lambda_1$, i.e. to $\lambda^2\mu=-1$. So 
we will assume $\mu\neq -1$ henceforth without loss of generality.

We focus on the following equation
\begin{equation}\label{ppequation}
[A,B][P,D]=\xi_{\mu}.
\end{equation}
Our purpose is to find the solutions of (\ref{ppequation}). 

Consider the following invariants:
 \begin{eqnarray*}
 t_1=\tr([A,B]), \\
 t_2=\tr([P,D]).
 \end{eqnarray*}

Denote by $\eta=[A,B]=\left(\begin{array}{cc} a & b\\ c &d\end{array}\right)$ and $\delta=[P,D]$ so that 
  $$
 \delta=\eta^{-1}\,\xi_{\mu} =\left(\begin{array}{cc}d\mu& -b\mu^{-1} \\-c\mu &a\mu^{-1} \end{array}\right).
  $$
Now $t_1=a+d$ and $t_2=a\mu^{-1}+d\mu$, so
  \begin{equation}\label{eq:eta}
  \left\{ \begin{array}{l} 
  a=\displaystyle\frac{\mu t_1-t_2}{\mu-\mu^{-1}}\\[10pt]
  d=\displaystyle\frac{t_2-\mu^{-1}t_1}{\mu-\mu^{-1}} \\[10pt]
  ad-bc=1\end{array}\right.
  \end{equation}

Consider the following set of matrices $P\in \SL(2,\CC)/\CC^{\ast}$,
$$
\cP=\left\{P \,|\, [P,D]=\delta=\left(\begin{array}{cc}d\mu & -b\mu^{-1} \\-c\mu &a\mu^{-1}\end{array}\right)
, a,b,c,d \text{ satisfy } (\ref{eq:eta}) \right\} \subset \SL(2,\CC)/\CC^{\ast}.
$$

Denote $P=\left(\begin{array}{cc}x&y\\z&w\end{array}\right)$, where $xw-yz=1$.
Also
\begin{equation*}
[P,D]=\left(\begin{array}{cc}x&y\\z&w\end{array}\right)\left(\begin{array}{cc}\lambda &0\\0&\lambda^{-1}\end{array}\right)\left(\begin{array}{cc}w&-y\\-z&x\end{array}\right)\left(\begin{array}{cc}\lambda^{-1}&0\\0&\lambda\end{array}\right)=\left(\begin{array}{cc} xw-\lambda^{-2}yz & xy(1-\lambda^{2})\\zw (1-\lambda^{-2}) & xw-\lambda^{2}yz\end{array}\right).
\end{equation*}

Hence, equality $[P,D]=\delta$ is now
\begin{equation}\label{eq:PD}
\left(\begin{array}{cc} xw-\lambda^{-2}yz & xy(1-\lambda^{2})\\zw (1-\lambda^{-2}) & xw-\lambda^{2}yz\end{array}\right)=\left(\begin{array}{cc}d\mu & -b\mu^{-1} \\-c\mu &a\mu^{-1}\end{array}\right).
\end{equation}

It gives us the following equations:
\begin{eqnarray}
xw-\lambda^{-2}yz&=&d\mu \label{eq:PD1}\\
xy(1-\lambda^{2})&=&-b\mu^{-1} \label{eq:PD2}\\
zw(1-\lambda^{-2})&=&-c\mu\label{eq:PD3}\\
xw-\lambda^{2}yz&=&a\mu^{-1}\label{eq:PD4} \\
xw-yz &= &1.\label{eq:detP}
\end{eqnarray}

\begin{lem} \label{lem:rel-t1-t2}
If (\ref{eq:PD1})--(\ref{eq:detP}) have solutions then 
 \begin{equation} \label{eq:traces}
 \mu(\lambda^2-1) t_1 + (1-\mu^2\lambda^2) t_2= (1-\mu^2)(1+\lambda^2).
   \end{equation}
This is actually equivalent to $(a\mu^{-1}-1) + \lambda^2(d\mu-1)=0$.
\end{lem}

\begin{proof}
From (\ref{eq:PD1}) and (\ref{eq:detP}), we have
 $yz= \frac{d\mu-1}{1-\lambda^{-2}}$. 
From (\ref{eq:PD4}) and (\ref{eq:detP}), we have
 $yz= \frac{a\mu^{-1}-1}{1-\lambda^2}$.
Equating both, we get 
  \begin{equation*}% \label{eq:yz}
   \frac{d\mu-1}{1-\lambda^{-2}} = \frac{a\mu^{-1}-1}{1-\lambda^2}
     \end{equation*}
that is rewritten as 
  \begin{equation} \label{eq:yz}
  (a\mu^{-1}-1) + \lambda^2(d\mu-1)=0.
     \end{equation}
Now using (\ref{eq:eta}), we get the result.
\end{proof} 

Lemma \ref{lem:rel-t1-t2} says that we can write $t_1$ in terms of $t_2$ as
\begin{equation*}\label{eq:t1t2}
 t_1=\frac{\mu^2\lambda^{2}-1}{\mu(\lambda^{2}-1)}t_2+ \frac{(1-\mu^{2})(1+\lambda^{2})}{\mu(\lambda^{2}-1)}.
\end{equation*}
Therefore we have a projection 
  \begin{eqnarray}\label{eqn:18pi}
  \pi: \cP  & \too & \CC  \notag \\
    P & \mapsto & t_2=\Tr ([P,D]).
   \end{eqnarray}

\begin{lem}\label{lem:bc=0}
Assuming (\ref{eq:traces}) holds, the condition $bc=0$  is equivalent to one of the following:
\begin{itemize} 
\item $a= d^{-1}=\mu$, $\delta=\left( \begin{array}{cc}1 & 0\\ -c\mu & 1\end{array}\right)$ or 
$\left(\begin{array}{cc} 1& -b\mu^{-1}\\0&1\end{array}\right)$. Here $t_1=\mu+\mu^{-1}$ and $t_2=2$,
\item  $a=d^{-1}= \mu\lambda^{2}$, $\delta=\left(\begin{array}{cc}\lambda^{2} & 0\\ -c\mu& \lambda^{-2}\end{array}\right)$
or $\left(\begin{array}{cc} \lambda^{2}& -b\mu^{-1} \\0&\lambda^{-2}\end{array}\right)$. Here
$t_1=\mu\lambda^2+\mu^{-1} \lambda^{-2}$ and $t_{2}=\lambda^{2}+\lambda^{-2}$.
\end{itemize}
\end{lem}

\begin{proof}
The equality $bc=0$ means that the matrix $\delta$ in (\ref{eq:PD}) is equal to 
$\left(\begin{array}{cc} d\mu & 0\\-c\mu & a\mu^{-1}\end{array}\right)$ or 
$\left(\begin{array}{cc} d\mu & -b\mu^{-1}\\0& a\mu^{-1}\end{array}\right)$, depending
on whether $b=0$ or $c=0$ (or both). Note that in this case $ad=1$, so $a=d^{-1}$.
By Lemma \ref{lem:rel-t1-t2} (see equation (\ref{eq:yz})), we have
\begin{eqnarray}
 &&a\mu^{-1}-1=(1-d\mu)\lambda^{2} \notag \\
  && \Longleftrightarrow \mu^{-1} a+\lambda^{2}d\mu-\lambda^2=1 \notag \\
 && \label{eqn:bc}
  \Longleftrightarrow (\mu d-1)(\mu^{-1}a-\lambda^{2})=ad-1=bc=0.
\end{eqnarray}
Hence it must be $d=\mu^{-1}$ or $a=\mu\lambda^2$. In the first case,
$a=\mu$, $t_1=\mu+\mu^{-1}$ and $t_2=a\mu^{-1}+d\mu=2$. In the second case,
$a=\mu \lambda^2$, $t_1=\mu\lambda^2+\mu^{-1} \lambda^{-2}$ and $t_{2}=\lambda^{2}+\lambda^{-2}$.
\end{proof}

\begin{rmk} 
Note that $\lambda^2+\lambda^{-2}\neq 2$, as $\lambda\neq \pm 1$.
\end{rmk}

\begin{thm}\label{thm:P}
Consider $\cP_{t_2}=\pi^{-1}(t_2)$, where $\pi$ is given in (\ref{eqn:18pi}), and $(t_1,t_2)$ satisfy
(\ref{eq:traces}). Then
\begin{itemize}
\item[(1)] If $t_2 \in \CC-\{2,\, \lambda^{2}+\lambda^{-2}\}$ then
$$
\cP_{t_2}\cong\left\{(x,y,z,w,b,c)\,|\, x\neq 0,\, y=\frac{-\mu^{-1} b}{x(1-\lambda^2)}, \, 
z=\frac{ (\mu-a)x}{ b},\,w=\frac{\mu^{-1}a-\lambda^{2} }{( 1- \lambda^{2})x}
, bc=k \right\}/\CC^\ast
$$
for some $k=k(t_2)\neq 0$.
\item[(2)] If $t_2=2$ then 
$$
\cP_{t_2}\cong\left\{(x,y,z,w,b,c)\,|\, x\neq0,\, y=\frac{-\mu\,b}{x(1-\lambda^{2})},\, 
z= \frac{-\mu^{-1} c}{(1-\lambda^{-2})w} \, , w=\frac1x\, ,bc=0 \right\}/\CC^\ast
$$
\item[(3)] If $t_2=\lambda^2+\lambda^{-2}$ then 
$$
\cP_{t_2}\cong\left\{(x,y,z,w,b,c)\,|\, x= \frac{-\mu\,b}{y(1-\lambda^{2})}  ,\, y\neq 0,\, 
z= -\frac1y,\, w= \frac{-\mu^{-1} c}{(1-\lambda^{-2})z}\, ,bc=0 \right\}/\CC^\ast
$$
\end{itemize}
where $\CC^\ast$ acts as $(x,y,z,w)\mapsto (\alpha x, \alpha^{-1}y, \alpha z, \alpha^{-1}w)$.
So $\cP_{t_2}\cong \CC^\ast$ in (1), and $\cP_{t_2} \cong \{bc=0\}$ in (2) and (3). 
\end{thm}

\begin{proof}
To prove $(1)$, note that Lemma \ref{lem:bc=0} implies that $bc\neq 0$.
Now $a,d$ are fixed by (\ref{eq:eta}) and $bc=ad-1=k\neq 0$.

We show that $\cP_{t_2}\subset 
\left\{(x,y,z,w)\in \CC^4\,|\, x\neq 0,\, y=\frac{-\mu b}{x(1-\lambda^2)}, \, 
z=\frac{ (1-\mu^{-1}a)x}{\mu^2 b},\,w=\frac{a-\mu\lambda^{2}}{( \lambda^{2}-1)x}\right\}$. 
First (\ref{eq:PD2}) implies $xy=\frac{-\mu b}{1-\lambda^{2}}$, and as $x\neq 0$ (since $b\neq 0$) then we get 
$y=\frac{-\mu b}{x(1-\lambda^{2})}$. Now equations (\ref{eq:PD4}) and (\ref{eq:detP}) imply that
$yz=\frac{\mu^{-1}\,a-1}{1-\lambda^{2}}$. We can divide by $y$ since $y\neq 0$, 
hence $z=\frac{\mu^{-1}\,a-1}{y(1-\lambda^{2})}=\frac{ (1-\mu^{-1}a)x}{\mu b}$.
Finally equation (\ref{eq:PD3}) implies $zw=\frac{-\mu^{-1} c}{1-\lambda^{-2}}$ hence, dividing by $z$ (since $c\neq0$), 
$w=\frac{-\mu^{-1} c}{z(1-\lambda^{-2})} =
\frac{- bc}{( 1-\lambda^{-2})(1-\mu^{-1}a)x}$.
Using that $bc=(\mu d-1)(\mu^{-1}a-\lambda^{2})$ (equation (\ref{eqn:bc})) 
and $d\mu-1=- \lambda^{-2}(a\mu^{-1}-1)$ (equation (\ref{eq:yz})),
we get $w= \frac{\mu^{-1} a-\lambda^{2}}{( 1- \lambda^{2})x}$.

For the reverse inclusion, take $x\neq 0,\, y=\frac{-\mu b}{x(1-\lambda^2)}, \, 
z=\frac{ (1-\mu^{-1}a)x}{\mu b},\,w=\frac{\mu^{-1}a-\lambda^{2}}{( 1-\lambda^{2})x}$. Then clearly (\ref{eq:PD2}) holds. Now
\begin{eqnarray*} 
 xw &=& \frac{\mu^{-1}a-\lambda^2}{1-\lambda^2}, \\
 yz &=& \frac{\mu^{-1}a-1}{1-\lambda^2}.
\end{eqnarray*}
So (\ref{eq:PD1}), (\ref{eq:PD4}) and (\ref{eq:detP}) are
\begin{eqnarray*} 
 xw -yz &=&  1, \\
 xw -\lambda^2 yz &=& \mu^{-1}a, \\
 xw -\lambda^{-2} yz &=& 1- \lambda^{-2}(a\mu^{-1}-1)=d\mu .
\end{eqnarray*}
Finally, 
 $$
 zw=
 \frac{ (1-\mu^{-1}a) (\mu^{-1}a -\lambda^{2})  }{\mu b (1- \lambda^{2})}=
 \frac{ \lambda^2(d \mu-1) (\mu^{-1}a -\lambda^{2})  }{\mu b (1- \lambda^{2})}=
  \frac{\lambda^2 bc }{\mu b(1-\lambda^2)}= -\frac{c\mu^{-1}}{1-\lambda^{-2}},
 $$
using (\ref{eqn:bc}) and (\ref{eq:yz}) again.

To prove (2), note that by Lemma \ref{lem:bc=0}, $a=\mu$. Then $yz=\frac{1-a\mu^{-1}}{1-\lambda^2}=0$. So
$xw=1$. The formulas for $y$ and $z$ follow straight away. The reverse inclusion is equally easy.

To prove (3), note that Lemma \ref{lem:bc=0} says that $a=\mu\lambda^2$. From
(\ref{eq:PD4}) and (\ref{eq:detP}), $xw=\frac{a\mu^{-1} - \lambda^2}{1-\lambda^2}=0$. 
So $yz=-1$. The formulas for $x$ and $w$ follow clearly,
and the reverse inclusion is easy.
\end{proof}

Now we compute the Hodge-Deligne polynomials.

\subsection{Case $\lambda^{2}\mu\neq \pm1$.} \label{case:1}
There are several contributions depending on the values of the parameter $t_2$.

\begin{itemize}
\item When $t_{2}=2$ then $t_{1}=\mu+\mu^{-1}$. Hence  $[P,D]=\left(\begin{array}{cc}\mu^{-1}d &\ast \\ \ast &\mu a 
  \end{array}\right)$ and 
$[A,B]=\left(\begin{array}{cc}\mu& b\\c& \mu^{-1} \end{array}\right)
\sim \left(\begin{array}{cc}\mu& 0\\0& \mu^{-1} \end{array}\right)$, as $bc=0$.
The contribution is $\overline{X}_{4,\mu}$ from $[A,B]$,
and from $[P,D]$ we get $\{bc=0\}$ by Theorem \ref{thm:P}(2). 

The contribution of this fiber is thus
\begin{align*}
e(F_1) &=(2q-1)e(\overline{X}_{4,\mu})\\ 
&=(2 q-1)(q^3+ 3 q^2 -3q-1 )\\
&= 2q^4 +5q^3-9 q^2+q+1.
\end{align*}

\item When $t_2=\lambda^{2}+\lambda^{-2}$, Lemma \ref{lem:rel-t1-t2} says that
$t_{1}=\mu \lambda^{2}+ \mu^{-1}\lambda^{-2} \neq \pm 2$ (here we use $\lambda^{2}\mu\neq \pm 1$). 
The contribution is then, by Theorem \ref{thm:P}(3),
\begin{align*}
e(F_2) &=(2q-1)e(\overline{X}_{4,\mu\lambda^{2}}) \\
&= 2q^4 +5q^3-9 q^2+q+1.
\end{align*}

\item When $t_1=2$ then $bc\neq 0$ hence the matrix $\eta\neq \Id$. So $\eta \sim J_{+}$. 
It cannot be $t_2=2$ or $t_2=\lambda^{2}+\lambda^{-2}$, as in these cases $t_1\neq 2$.
So Theorem \ref{thm:P}(1) says that the contribution of $\cP$ is $\CC^\ast$. Hence  
\begin{align*}
 e(F_3) &=(q-1)e(\overline{X}_{2}) \\ 
 & =(q-1)(q^3-2q^2-3 q ) \\
 &=q^4-3q^3-q^2+3 q .
\end{align*}

\item The case $t_1=-2$ is analogous, the only difference being that the matrix $\eta$ is of Jordan type $J_{-}$. 
The contribution is then
\begin{align*}
 e(F_4) &=(q-1) e(\overline{X}_{3}) \\ 
 &=(q-1)(q^3+3 q^2)\\
 &= q^4 +2 q^3- 3q^2.
\end{align*}

\item The generic case is a fibration with base $t_1\in L=\CC-\{\pm 2,\mu+\mu^{-1},
\lambda^{2}\mu+\lambda^{-2}\mu^{-1}\}$. The fibration $\cP\to L$ is trivial with fibers
$\CC^\ast$. The fibration $\{(A,B)\} \to L$ has total space $\overline{X}_4/\ZZ_2$, and
we remove two fibers. Therefore
\begin{align*}
 e(F_5)&=(q-1)
\left(e(\overline{X}_4/\ZZ_{2})-e(\overline{X}_{4,\mu})-e(\overline{X}_{4,\lambda^2\mu})\right)=\\
 &=(q-1) (q^4-2q^3-3q^2+3q+1) - 2 (q^3 + 3 q^2 - 3 q-1))\\
   &=q^5 - 5 q^4 - 5 q^3 + 18 q^2 - 6 q -3.
\end{align*}

\end{itemize}

The total sum of all contributions is
$$
e( \overline{Z}{}_{44}^{\lambda_1\lambda_2})=e(F_1)+e(F_2)+e(F_3)+e(F_4)+e(F_5)=  q^5+ q^4+ 4q^3-4q^2-q-1.
$$
We quotient by $\CC^\ast$ to get the sought moduli space. So
$$
 e(\cR_{\xi_{\lambda_1},\xi_{\lambda_2}})= e( \overline{Z}{}_{44}^{\lambda_1\lambda_2}/\CC^\ast)= q^4+2q^3+6q^2+2q +1.
$$

\subsection{Case $\lambda^{2}\mu = - 1$} \label{case:2}

\begin{itemize}
\item The set $F_1$ is defined by $t_2=2$ and $t_1=\mu+\mu^{-1}$. As in the previous case,
we have $e(F_1)= 2q^4 +5q^3-9 q^2+q+1$.
\item The set $F_2$ is defined by $t_2=\lambda^{2}+\lambda^{-2}$ and $t_1=-2$. We 
have $bc=0$ now, so $\cP$ is computed in Theorem \ref{thm:P}(3) as $\{bc=0\}$. There are
three cases:
\begin{itemize}
\item $b=0,\,c\neq 0$. This gives $(q-1)e(\overline{X}_3)$, since $[A,B]\sim J_{-}$,
\item $b\neq 0,\, c=0$. This gives $(q-1)e(\overline{X}_3)$, since $[A,B]\sim J_{-}$,
\item $b=c=0$. This gives $e(X_1)$, since $[A,B]= -\Id$.
\end{itemize}
So the sum is 
\begin{align*}
 e(F_2) &=2(q-1) e(\overline{X}_3)+ e(X_1) \\
 & =(q^3-q) + 2 (q-1 ) (q^3+3 q^2 ) \\
 &=2q^4 + 5q^3-6q^2 -q .
 \end{align*}
\item The set $F_3$ is defined by $t_1=2$. We note that $bc\neq 0$, so $\eta\sim J_+$.
The contribution is 
\begin{align*}
 e(F_3) &=(q-1) e(\overline{X}_{2})\\ 
 &=q^4-3q^3-q^2+3 q .
\end{align*}
\item Finally,
the generic case is a fibration with base $t_1\in L=\CC-\{\pm 2,\mu+\mu^{-1}\}$. 
The fibration $\cP\to L$ is trivial with fibers
$\CC^\ast$. The fibration $\{(A,B)\}\to L$ has total space $\overline{X}_4/\ZZ_2$, and
we remove one fiber. Therefore
\begin{align*}
 e(F_4)&=(q-1) \left(e(\overline{X}_4/\ZZ_{2})-e(\overline{X}_{4,\mu})\right) \\
 &=q^5 -4q^4-3q^3+12 q^2-4q-2.
\end{align*}
\end{itemize}

The total sum is
 $$
 e( \overline{Z}{}_{44}^{\lambda_1\lambda_2})=q^5+ q^4 +4 q^3- 4 q^2-q-1  
  $$
and 
$$
 e(\cR_{\xi_{\lambda_1},\xi_{\lambda_2}})= e( \overline{Z}{}_{44}^{\lambda_1\lambda_2}/\CC^\ast)= q^4+2q^3+6q^2+2q +1.
$$

%%%%%%%%%%%%%%%%%%%%%%%%%%%%%%%%%%%%%%%%%%%%%%%%%%%%%%%%%%
\section{The case $\lambda_1=\lambda_2$}
%%%%%%%%%%%%%%%%%%%%%%%%%%%%%%%%%%%%%%%%%%%%%%%%%%%%%%%%%%

We end up with the moduli space $\cR_{\xi_{\lambda_1},\xi_{\lambda_1}}$,
in which case the moduli space has reducibles. For convenience, we are going to take
$\lambda_2=\lambda_1^{-1}$, since $\xi_{\lambda_2} \sim \xi_{\lambda_1}$.
Then $\mu=\lambda_1^{-1}\lambda_2^{-1}=1$. 
Now we have 
 \begin{align*}
 Z_{44}^{\lambda_1\lambda_2} &= \{(A,B,C_1,C_2) | C_1,C_2\in W_{4,\lambda}, [A,B]C_1C_2=\Id\}, \\
 \overline{Z}{}_{44}^{\lambda_1\lambda_2} &= \{(A,B,P) | [A,B][P,D]=\Id\}/\CC^*
 \end{align*}
where $D=\left(\begin{array}{cc} \lambda & 0 \\ 0 & \lambda^{-1} \end{array}\right)$. Here
$\lambda=\lambda_1$ 
and $\CC^*$ acts on $P$ by multiplication on the right.
Also $t_1=\Tr([A,B])$, $t_2=\Tr([P,D])$ satisfy $t_2=t_1$.

Denote by $[A,B]=\left(\begin{array}{cc} a & b\\ c &d\end{array}\right)$ and 
$$
\cP=\left\{P \,|\, [P,D]=\left(\begin{array}{cc}d & -b \\-c &a\end{array}\right)\right\} \subset \SL(2,\CC)/\CC^*.
$$
The computations in Section \ref{sec:diag} remain valid with $\mu=1$ except
that we cannot use equation (\ref{eq:eta}). This is only used to 
reduce equation (\ref{eq:yz}). But note that this equation can be
reduced directly by using $t_1=a+d$, $t_2=a\mu^{-1}+d\mu$.
In particular, Theorem \ref{thm:P} also holds for $\mu=1$.
The computation of the Hodge-Deligne polynomials of the strata now runs as follows.

\begin{itemize}
\item When $t_2=2$, we have the subset $F_1$ is as in the first case of Subsection \ref{case:1}. 
So $e(F_1)= 2q^4 +5q^3-9 q^2+q+1$.
\item When $t_2=\lambda^{2}+\lambda^{-2}=t_1$, we have $bc=0$. 
There are three cases:
 \begin{itemize}
  \item $b=0$, $c\neq 0$. This gives $(q-1) e(\overline X_2)$, since $[A,B]\sim J_+$,
  \item $b \neq 0$, $c= 0$. This gives $(q-1) e(\overline X_2)$, since $[A,B]\sim J_+$,
  \item $b=c=0$. This gives $e(X_0)$, since $[A,B]=\Id$.
  \end{itemize}
So the sum is 
$$
e(F_2)=2(q-1) e(\overline X_2)+ e(X_0)=3q^4-2q^3-3q^2+2q.
$$
\item The stratum defined by $t_1=-2$ contributes
\begin{align*}
 e(F_3) &=(q-1) e(\overline{X}_{3}) \\ 
 &= q^4 +2 q^3- 3q^2.
\end{align*}
\item The remaining contribution 
is a fibration with base $t_1\in L=\CC-\{\pm 2,\mu+\mu^{-1}\}$. So,
as in the fourth case of Subsection \ref{case:2}, we have
\begin{align*}
 e(F_4)&=(q-1) \left(e(\overline{X}_4/\ZZ_{2})-e(\overline{X}_{4,\mu})\right) \\
 &=q^5 -4q^4-3q^3+12 q^2-4q-2.
\end{align*}
\end{itemize}
The total sum is
  $$ 
  e( \overline{Z}{}_{44}^{\lambda_1\lambda_2})=q^5+ 2q^4 +2q^3 -3 q^2 -q -1 \,.
  $$

To finish, we compute the Hodge-Deligne polynomial of 
 $$
 \cR_{\xi_\lambda,\xi_\lambda}= Z_{44}^{\lambda_1\lambda_2}//\PGL(2,\CC) =\overline{Z}{}_{44}^{\lambda_1\lambda_2}//\CC^*.
 $$
The reducibles  $(A,B,C_1,C_2)\in Z_{44}^{\lambda_1\lambda_2}$ satisfy
that $[A,B]=\Id$, $C_1=C_2^{-1}\in W_{4,\lambda}$, and the stabilizer of $C_1$ also stabilize $A,B$.
%Let $\cD\subset  Z_{44}^{\lambda_1\lambda_2}$ 
%be the space of reducibles, and $Z_{44}^{\lambda_1\lambda_2*}= Z_{44}^{\lambda_1\lambda_2}-\cD$.
Correspondingly, the reducibles  $(A,B,P)\in  \overline{Z}{}_{44}^{\lambda_1\lambda_2}$ satisfy
that $A,B,P$ are diagonal matrices and  $C=[P,D]=[A,B]^{-1}=\Id$. 
%Let $\overline\cD$ be the subset of reducibles
%and  $\overline{Z}{}_{44}^{\lambda_1\lambda_2*}=\overline{Z}{}_{44}^{\lambda_1\lambda_2}-\overline\cD$.
We also have to see which orbits have reducibles in their closure under the action
of the group $\CC^*$ of diagonal matrices. For this, $A,B$ should be
either upper triangular or lower triangular.
The matrices of the form $A=\left(\begin{array}{cc} x & a \\ 0 & x^{-1}\end{array}\right)$,
$B=\left(\begin{array}{cc} y & b \\ 0 & y^{-1}\end{array}\right)$, $C=[A,B]^{-1}$,
converge to 
$A'=\left(\begin{array}{cc} x & 0 \\ 0 & x^{-1}\end{array}\right)$,
$B'=\left(\begin{array}{cc} y & 0 \\ 0 & y^{-1}\end{array}\right)$, $C'=\Id$. 
Hence they go to a single orbit in the GIT quotient. Let $\overline\cD$ be the space
of $A,B$ both either upper-triangular or lower-triangular. Then $e(\overline\cD)=(q-1)^2
(2(q-1)^2+4(q-1)+1)=(q-1)^2(2q^2-1)$. So the irreducibles
$\overline{Z}{}_{44}^{\lambda_1\lambda_2*}=\overline{Z}{}_{44}^{\lambda_1\lambda_2}-\overline\cD$ have 
 $$
 e(\overline{Z}{}_{44}^{\lambda_1\lambda_2*})=q^5+6q^3-4q^2-3q.
 $$
In the quotient, the contribution of the reducibles is
$(q-1)^2$. Hence
  \begin{align*} 
 e(\cR_{\xi_\lambda,\xi_\lambda})&=e(\overline{Z}{}_{44}^{\lambda_1\lambda_2*}/\CC^*) + (q-1)^2 \\
 &= q^4+q^3+8q^2+q+1.
  \end{align*}

%%%%%%%%%%%%%%%%%%%%%%%%%%%%%%%%%%%%%%%%%%%%%%%%%%%%%%%%%%
\section{Betti numbers and Hodge numbers} \label{sec:hodge}
%%%%%%%%%%%%%%%%%%%%%%%%%%%%%%%%%%%%%%%%%%%%%%%%%%%%%%%%%%

With the simultaneous information of the Hodge-Deligne polynomial and the Poincar\'e polynomial
we are able to recover some of the Hodge numbers of the spaces. 
Consider the variety 
$Z=\cR_{\xi_\lambda,\xi_\mu}(X, \SL(2,\CC))$, $\lambda\neq \mu^{\pm 1}$, 
whose Hodge-Deligne polynomial was computed in Section \ref{sec:diag},
$$
e(\cR_{\xi_\lambda,\xi_\mu}(X, \SL(2,\CC))) =q^4+2q^3+6q^2+2q+1.
$$
Recall that the Hodge-Deligne polynomial is
$e(Z)=\sum_{p}\sum_{k} (-1)^{k}h_{c}^{k,p,p}(Z) q^p = \sum_{p} e_p \, q^p$.

For $Z=\cR_{\xi_\lambda,\xi_\mu}(X, \SL(2,\CC))$, and $|\lambda|=|\mu|=1$, the Poincar\'e polynomial was computed in \cite{BY},
$$
P_t(\cR_{\xi_\lambda,\xi_\mu}(X, \SL(2,\CC)))=10 t^4+2 t^3+3 t^2+1.
$$
Now note that the family $\cR_{\xi_\lambda,\xi_\mu}(X, \SL(2,\CC))$ over $\{(\lambda,\mu)\in (\CC^*-\{\pm 1\})^2|\lambda\neq \mu,\mu^{-1}\}$ is
analytically locally trivial. In particular, the spaces are diffeomorphic, and hence all have the same Poincar\'e polynomial,
for any $\lambda  \neq \mu,\mu^{-1}$, $\lambda,\mu\neq 0,\pm 1$.

This character variety is smooth and of complex dimension $4$ so its compactly supported Poincar\'e polynomial is 
$$
P_{t}^{c}(\cR_{\xi_\lambda,\xi_\mu}(X, \SL(2,\CC)))=t^{8} + 3 t^{6} + 2 t^{5} + 10 t^{4}.
$$

We write  $h^{k,p,p}_{c}=h^{k,p,p}_{c}(\cR_{\xi_\lambda,\xi_\mu}(X, \SL(2,\CC)))$ and 
$b^{k}_{c}=b^{k}_{c}(\cR_{\xi_\lambda,\xi_\mu}(X, \SL(2,\CC)))$.
Note that the elements in the columns of Table \ref{tab} add as 
$b^{k}_{c}=\sum_{p} h^{k,p,p}_{c}$, and the elements of the rows satisfy that their alternate sums are
$e_p=\sum_{k}(-1)^{k}h_c^{k,p,p}$. So we have 
\begin{table} [H]
\begin{tabular}{|c|c|c|c|c|c|c|}
    \hline
    ~             & $b^{4}_{c}=10$ & $b^{5}_{c}=2$ & $b^{6}_{c}=3$ & $b^{7}_{c}=0$ & $b^{8}_{c}=1$ & $\sum_{k}(-1)^{k}h^{k,p,p}_{c}$ \\ \hline
    $h^{k,0,0}_{c}$ & $h^{4,0,0}_{c}=1+a-\epsilon_1$      & $h^{5,0,0}_{c}=a$     & $h^{6,0,0}_{c}=\epsilon_1$      & $0$          & $0$         & $1$       \\ \hline
    $h^{k,1,1}_{c}$ & $h^{4,1,1}_{c}=2+b-\epsilon_2$       &  $h^{5,1,1}_{c}=b$    & $h^{6,1,1}_{c}=\epsilon_2$      & $0$         & $0$        & $2$        \\ \hline
    $h^{k,2,2}_{c}$ & $h^{4,2,2}_{c}=6+c-\epsilon_3$      & $h^{5,2,2}_{c}=c$     & $h^{6,2,2}_{c}=\epsilon_3$   & $0$       & $0$         & $6$        \\ \hline
    $h^{k,3,3}_{c}$ & $0$                   & $0$                 & $2$                   & $0$                  & $0$                 & $2$                         \\ \hline
    $h^{k,4,4}_{c}$ & $0$                   & $0$                 & $0$                   & $0$                  & $1$                 & $1$                         \\ \hline
    \end{tabular}
     \caption {}\label{tab}
\end{table}
\noindent where $\epsilon_1,\epsilon_2,\epsilon_3$ are $1,0,0$ or $0,1,0$ or $0,0,1$; and $a,b,c$ are positive
integers such that $a+b+c=2$. This means that there are $18$ possible solutions.
In any case, at least we get that 
$h^{7,p,p}_{c}(Z)=0$ for all $p=0,1,2,3,4$, $h^{8,p,p}_{c}(Z)=0$ for $p=0,1,2,3$, 
$h^{8,4,4}_{c}=1$, $h^{k,3,3}_{c}(Z)=0$ for $k=4,5,7,8$ and $h^{6,3,3}_{c}(Z)=2$.

\end{document}